\documentclass[12pt]{article}
\usepackage[utf8]{inputenc}
\usepackage{amsmath}
\usepackage{amsfonts}
\usepackage{amsthm}
\usepackage[hang, flushmargin]{footmisc}

\def\modd#1 #2{#1\ \mbox{\rm (mod}\ #2\mbox{\rm )}}

\theoremstyle{plain}
\newtheorem{theorem}{Theorem}
\newtheorem{corollary}[theorem]{Corollary}
\newtheorem{lemma}[theorem]{Lemma}
\newtheorem{proposition}[theorem]{Proposition}

\theoremstyle{definition}

\newtheorem{example}[theorem]{Example}
\newtheorem{conjecture}[theorem]{Conjecture}

\theoremstyle{remark}

\providecommand{\keywords}[1]{\small \textit{Keywords:} #1}

\providecommand{\MSC}[1]{\small \textit{MSC 2020:} #1}

\newcommand\blfootnote[1]{%
  \begingroup
  \renewcommand\thefootnote{}\footnote{#1}%
  \addtocounter{footnote}{-1}%
  \endgroup
}

\begin{document}

\begin{center}
\vskip 1cm{\LARGE\bf 
A Perfect Number Generalization  \\
\vskip .1in
and Some Euclid-Euler Type Results}
\vskip 1cm
\large
Tyler Ross\\
School of Mathematical Sciences\\
Zhejiang University\\
Hangzhou\\
310058 China\\
tylerxross@gmail.com \blfootnote{\noindent This project was supported by the Natural Science Foundation of China (Grant No.\ 12071421). } 
\end{center}

\begin{abstract}
    In this paper, we introduce a new generalization of the perfect numbers, called $\mathcal{S}$-perfect numbers. Briefly stated, an $\mathcal{S}$-perfect number is an integer equal to a weighted sum of its proper divisors, where the weights are drawn from some fixed set $\mathcal{S}$ of integers. After a short exposition of the definitions and some basic results, we present our preliminary investigations into the $\mathcal{S}$-perfect numbers for various special sets $\mathcal{S}$ of small cardinality. In particular, we show that there are infinitely many $\{0, m\}$-perfect numbers and $\{-1,m\}$-perfect numbers for every $m \geq 1$. We also provide a characterization of the $\{-1,m\}$-perfect numbers of the form $2^kp$ ($k \geq 1$, $p$ an odd prime), as well as a characterization of all even $\{-1, 1\}$-perfect numbers.
\end{abstract}
\\
\MSC{11A25, 11A67, 11Y55}

\section{Introduction} \label{sec1}

A positive integer $n > 1$ is called a \textit{perfect number} if it is equal to the sum of its proper divisors; symbolically, if
\begin{equation}\nonumber
n=\sum_{\substack{1 \leq d < n \\ d|n}}d. 
\end{equation}
It has been known since Euclid that any number of the form $n=2^{p-1}(2^p-1)$ where both $p$ and $2^p-1$ are prime is perfect. Centuries later, Euler proved the converse: if any $n > 1$ is an even perfect number, then it is of the form $n=2^{p-1}(2^p-1)$ with both $p$ and $2^p-1$ prime. On the other hand, it is not known if there exist infinitely many perfect numbers, or if there exists even a single odd perfect number. To date, fifty-one even perfect numbers have been found, many by computer search, the largest of which corresponds to the prime number $p = 82589933$ (see \cite{article2} for a survey of the current state of research). The literature has been broadened by the introduction of various generalizations, to which this paper adds another, encompassing many of those previously put forward.

Let $\mathcal{S} \subset \mathbb{Z}$ be any collection of integers, and let $n \in \mathbb{Z}$ with $|n| > 1$. Then we call $n$ an \textit{$\mathcal{S}$-perfect number of the first kind} if there exist integers $\lambda_1, \ldots, \lambda_k \in \mathcal{S}$ such that
\begin{equation}\nonumber
    1+\sum_{j=1}^{k}{\lambda_{j}d_{j}}=n  ,
\end{equation}
where $1 = d_0 < d_1 < \cdots < d_k < d_{k+1}=|n|$ are the positive divisors of $n$. We call $n$ an \textit{$\mathcal{S}$-perfect number of the second kind} if there exist integers $\lambda_0, \ldots, \lambda_k \in \mathcal{S}$ such that
\begin{equation}\nonumber
    \lambda_{0}+\sum_{j=1}^{k}{\lambda_{j}d_{j}}=n  .
\end{equation} 
If $n$ is an $\mathcal{S}$-perfect number, we refer to the sum $n=1+\sum\limits_{j=1}^{k}{\lambda_{j}d_{j}}$ (respectively $n=\lambda_0+\sum\limits_{j=1}^{k}{\lambda_{j}d_{j}}$) as an \textit{$\mathcal{S}$-presentation} of $n$, or simply a \textit{presentation} of $n$ when $\mathcal{S}$ is fixed.  

Throughout this paper we will limit our investigation to positive $\mathcal{S}$-perfect numbers of the first kind unless otherwise indicated. We prove that there are infinitely many $\{0, m\}$-perfect numbers and $\{-1,m\}$-perfect numbers for every $m \geq 1$. We also provide a characterization of the $\{-1,m\}$-perfect numbers of the form $2^kp$ ($k \geq 1$, $p$ an odd prime), as well as a characterization of all even $\{-1.1\}$-perfect numbers. The symbols $\tau$ and $\sigma$ indicate the familiar arithmetic functions
\nonumber \begin{align} 
    \tau(n) &= \sum_{d|n}1, \\
    \sigma(n) &= \sum_{d|n}d. 
\end{align}
\begin{example}The $\mathcal{S}$-perfect numbers generalize the perfect numbers ($\mathcal{S}=\{1\}$), as well several other previously defined generalizations of the perfect numbers. The $\{1,0\}$-perfect numbers (of the second kind) are the semiperfect numbers. For $k  \geq 1$, the $\{k\}$-perfect numbers of the first kind were introduced by Minoli and Bear \cite{article7} as $k$-hyperperfect numbers, and subsequently also studied by te Riele \cite{article9} and McCranie \cite{article5}. The $\{-k\}$-perfect numbers are integers $n < 0$ satisfying 
\begin{equation}\nonumber
    \sigma(|n|)=\frac{(k+1)(|n|+1)}{k}  .
\end{equation}
Such numbers have also been the object of some interest, for example in Guy \cite{article4} ($k=1$) and Bege and Fogarasi \cite{article1} ($k=2$).
\end{example}
\begin{example}
We list here the first few $\mathcal{S}$-perfect numbers for various small $\mathcal{S}$.
\begin{itemize}
    \item $\mathcal{S}=\{1\}:$ $6, 28, 496, 8128, 33550336, \ldots$ (perfect numbers, OEIS A005101)
    \item $\mathcal{S}=\{1,0\}$ (second kind)$:$ $6, 12, 18, 20, 24, 28, 30, 36, 40, 42, 48, 54, 56, 60, 66 \ldots$ (semiperfect numbers, OEIS A005835); the smallest semiperfect number that is not a $\{1,0\}$-perfect number of the first kind is $66$.
    \item $\mathcal{S}=\{2\}$$:$ $21, 2133, 19521, 176661, \ldots$ ($2$-hyperperfect numbers, OEIS A007593)
    \item $\mathcal{S}=\{3\}$$:$ $325, \ldots$ ($3$-hyperperfect numbers)
    \item $\mathcal{S}=\{0,2\}:$ $21, 63, 147, 171, 189, 225 \ldots$ (see Section \ref{sec2})
    \item $\mathcal{S}=\{-1,2\}$: $21, 28, 52, 84, 112, 156, 189, 208, 228, \ldots$ (see Section \ref{sec2}) 
    \item $\mathcal{S}=\{-1,1\}$$:$ $6, 12, 24, 28, 30, 40, 42, 48, 54, 56, 60, 66, 70, 78, 80, \ldots$ (see Section \ref{sec3})
    \item $\mathcal{S}=\{1,2\}$$:$ $6, 10, 21, 28, 44, 45, 50, 52, 99, 105, 117, 135, 136, \ldots$
    \item $\mathcal{S}=\{1,3\}$$:$ $6, 14, 15, 28, 44, 76, 110, 135, 152, 182, 184, 190, 231, \ldots$
    \item $\mathcal{S}=\{2,3\}$$:$ $21, 175, 325, 333, \ldots$.
\end{itemize}
\end{example}

The following proposition shows that for most integers $n > 1$ it is easy to find a set $\mathcal{S} \subset \mathbb{Z}$ such that $n$ is $\mathcal{S}$-perfect. For this reason we focus our discussion mainly on determining the $\mathcal{S}$-perfect numbers and related properties for fixed $\mathcal{S}$. This is somewhat at odds with the literature on $k$-hyperperfect numbers, in which the term hyperperfect number is used generically to refer to any integer that is $k$-hyperperfect for some $k \geq 1$.

\begin{proposition}
If $n \in \mathbb{Z}$ ($|n| > 1$) has at least two prime factors, then there exists a finite set $\mathcal{S} \subset \mathbb{Z}$ with $\#\mathcal{S} \leq \tau(n)-2$ such that $n$ is $\mathcal{S}$-perfect. If $n \in \mathbb{Z}$ is a prime power, then $n$ is not $\mathcal{S}$-perfect for any $\mathcal{S} \subset \mathbb{Z}$.
\end{proposition}
\begin{proof}
If $n$ has at least two prime factors, with positive divisors 
\begin{equation} \nonumber
1=d_0 < d_1 < \cdots < d_k < d_{k+1} = |n|,    
\end{equation}
then $\text{gcd}(d_1, \ldots, d_k)=1$. It follows that the linear diophantine equation
\begin{equation} \nonumber
    \sum_{j=1}^{k}d_{j}x_{j}=n-1
\end{equation}
has solutions. The second claim is obvious.
\end{proof}

For $\mathcal{S} \subset \mathbb{Z}$, we denote the set of $\mathcal{S}$-perfect numbers by $\mathcal{P}(\mathcal{S})$, omitting curly brackets when $\mathcal{S}$ is given by enumeration of its elements. We have the following easy inclusions.

\begin{proposition}
If $(\mathcal{S}_\alpha)_{\alpha \in \mathcal{A}}$ is any family of subsets $\mathcal{S}_{\alpha} \subset \mathbb{Z}$, then
\nonumber \begin{align}
    &\bigcup_{\alpha \in \mathcal{A}}\mathcal{P}(\mathcal{S}_\alpha) \subset \mathcal{P}(\bigcup_{\alpha \in \mathcal{A}}\mathcal{S}_\alpha), \\
    &\mathcal{P}(\bigcap_{\alpha \in \mathcal{A}}\mathcal{S}_\alpha) \subset \bigcap_{\alpha \in \mathcal{A}}\mathcal{P}(\mathcal{S}_\alpha) .
\end{align}
\end{proposition}
\begin{proof}
Follows immediately from the definitions.
\end{proof}

\section{Some special cases} \label{sec2}

In this section we investigate the $\{0,m\}$-perfect numbers and $\{-1,m\}$-perfect numbers for arbitrary $m \geq 1$. The former are dispatched quite easily via the following lemma.
\begin{lemma}
If $n \in \mathcal{P}(0, m)$ for some $m \geq 1$, then also $(m+1)n \in \mathcal{P}(0, m)$.
\end{lemma}
\begin{proof}
If $n = \sum$ is a $\{0, m\}$-presentation of $n$, then $(m+1)n=\sum + mn$ is a $\{0,m\}$-presentation of $(m+1)n$.
\end{proof}
Therefore it suffices to exhibit a single $n \in \mathcal{P}(0,m)$ to generate infinitely many $\{0,m\}$-perfect numbers, which gives the following theorem.
\begin{theorem}
There exist infinitely many $\{0,m\}$-perfect numbers for all $m \geq 1$.
\end{theorem}
\begin{proof}
Note that
\begin{equation} \nonumber
    (m+1)(m^2+m+1)=1+m(m+1)+m(m^2+m+1)
\end{equation}
is $\{0, m\}$-perfect for any $m \geq 1$.
\end{proof}

The $\{-1,m\}$-perfect numbers are more interesting. We focus on the $\{-1,m\}$-perfect numbers of the form $n=2^kp$ for some odd prime $p$. The following lemma and corollary will be useful for proving that there are infinitely many such numbers, and characterizing their occurrences among numbers of the same form. We make frequent use of the $2$-adic valuation $\nu_{2}(n) = \text{max}(k \geq 0: 2^{k} \text{ divides } n)$.
\begin{lemma} \label{lem2}
Let $0 \leq s \leq t$, $m \geq 1$. Then the numbers of the form $n=\sum\limits_{j=s}^t\lambda_j\cdot2^j$ with $\lambda_s, \ldots, \lambda_t \in \{-1,m\}$ are precisely the numbers $n \equiv \modd{-2^s(2^{t-s+1}-1)} {2^s(m+1)}$ in the interval
\begin{equation} \nonumber
    -2^s(2^{t-s+1}-1) \leq n \leq 2^sm(2^{t-s+1}-1).
\end{equation}
\end{lemma}
\begin{proof}
It is easy to see that
\begin{equation} \nonumber
    \sum_{j=s}^t \lambda_j \cdot 2^j \equiv \modd{-\sum_{j=s}^t 2^j = -2^s(2^{t-s+1}-1)} {2^s(m+1)}
\end{equation}
for any $\lambda_s, \ldots, \lambda_t \in \{-1,m\}$. Therefore the different choices of $\lambda_s, \ldots, \lambda_t \in \{-1,m\}$ give $2^{t-s+1}$ different numbers $n \equiv \modd{-2^s(2^{t-s+1}-1)} {2^s(m+1)}$ in the interval
\begin{equation} \nonumber
    -2^s(2^{t-s+1}-1) \leq n \leq \sum_{j=s}^t 2^jm = 2^sm(2^{t-s+1}-1).
\end{equation}
Since there are exactly $2^{t-s+1}$ such numbers, we are done.
\end{proof}
\begin{corollary} \label{cor1}
Fix $m \geq 1$ and set $\beta = \nu_{2}(m+1)$. If $n=\sum\limits_{j=s}^t \lambda_j \cdot 2^j$ for some $0 \leq s \leq t$, $\lambda_s, \ldots, \lambda_t \in \{-1,m\}$ with $t \geq s+\beta-1$, then also $n=\sum\limits_{j=s}^{t+\alpha} \Lambda_j \cdot 2^j$ for some $\Lambda_s, \ldots, \Lambda_{t+\alpha} \in \{-1,m\}$ whenever $2^\alpha \equiv \modd{1} {(m+1)/2^\beta}$.
\end{corollary}
\begin{proof}
If $2^\alpha \equiv \modd{1} {(m+1)/2^\beta}$, then $2^{t+\alpha+1} \equiv \modd{2^{t+1}} {2^{t+1-\beta}(m+1)}$. If $t+1-\beta \geq s$, then also $2^{t+\alpha+1} \equiv \modd{2^{t+1}} {2^s(m+1)}$, so $-2^s(2^{t-s+1}-1) \equiv \modd{-2^s(2^{t+\alpha-s+1}-1)} {2^s(m+1)}$, as required by the conditions in Lemma \ref{lem2}.
\end{proof}
\begin{theorem} \label{thrm1}
Fix $m \geq 1$ and set $\beta = \nu_{2}(m+1)$. If both $2^k p$, $2^{k+\alpha}p \in \mathcal{P}(-1,m)$ for some odd prime $p$ and $k, \alpha \geq 1$, then $2^\alpha \equiv \modd{1} {(m+1)/2^\beta}$. Conversely, if $2^k p \in \mathcal{P}(-1,m)$ for some odd prime $p$ and $k \geq \beta$, then also $2^{k+\alpha} p \in \mathcal{P}(-1,m)$ whenever $2^\alpha \equiv \modd{1} {(m+1)/2^\beta}$.
\end{theorem}
\begin{proof}
Suppose first that both $2^kp,$ and $2^{k+\alpha}p \in \mathcal{P}(-1,m)$, with presentations
\begin{align}
    \label{eqn1} \tag{1} 2^kp  &=  1+\sum_{j=1}^{k}\lambda^{(1)}_{j} \cdot 2^j + \sum_{j=0}^{k-1}\lambda^{(2)}_{j} \cdot 2^j p ,  \\
    \label{eqn2} \tag{2} 2^{k+\alpha}p  &= 1+\sum_{j=1}^{k+\alpha}\Lambda^{(1)}_{j} \cdot 2^j + \sum_{j=0}^{k+\alpha-1}\Lambda^{(2)}_{j} \cdot 2^j p
\end{align}
respectively. Note that every $\lambda^{(i)}_j, \Lambda^{(i)}_j \equiv \modd{-1} {m+1}$; we reduce the first equation to find \begin{equation} \nonumber
    2^kp  \equiv \modd{1-\sum_{j=1}^{k}2^j - \sum_{j=0}^{k-1} 2^j p} {m+1},
\end{equation}
or $(2^{k+1}-1)(p+1) \equiv \modd{2} {m+1}$, from which it follows easily that $p+1$ must be a unit modulo $(m+1)/2^\beta$. 

Subtracting \eqref{eqn1} from \eqref{eqn2} and reducing again modulo $m+1$ gives
\begin{equation} \nonumber
    2^kp(2^\alpha - 1) \equiv \modd{-\sum_{j=k+1}^{k+\alpha}2^j - \sum_{j=k}^{k+\alpha-1} 2^j p} {m+1},
\end{equation}
or $2^{k+1}(p+1)(2^{\alpha}-1) \equiv \modd{0} {m+1}$; so also $2^{k+1}(p+1)(2^{\alpha}-1) \equiv \modd{0} {(m+1)/2^\beta}$. Since both $2^{k+1}$ and $p+1$ are units modulo $(m+1)/2^\beta$, we conclude that $2^{\alpha}-1 \equiv \modd{0} {(m+1)/2^\beta}$.

Conversely, suppose $k \geq \beta$ and $2^k p \in \mathcal{P}(-1,m)$ for some odd prime $p$, with presentation
given by \eqref{eqn1}.
Let $2^\alpha \equiv \modd{1} {(m+1)/2^\beta}$. We have 
\begin{equation} \tag{3} \label{eqn3}
    2^{k+\alpha}p  =  1+\sum_{j=1}^{k}\lambda^{(1)}_{j} \cdot 2^j + \sum_{j=0}^{k-1}\lambda^{(2)}_{j} \cdot 2^j p + \sum_{j=k}^{k+\alpha-1}2^jp.
\end{equation}
Since $k \geq \beta$, we can use Corollary \ref{cor1} to find $\Lambda^{(1)}_{1}, \ldots, \Lambda^{(1)}_{k+\alpha}$ such that 
\begin{equation} \nonumber
    \sum_{j=1}^{k+\alpha}\Lambda^{(1)}_j \cdot 2^j =\sum_{j=1}^{k}\lambda^{(1)}_j \cdot 2^j.
\end{equation}
As for the remaining sum in \eqref{eqn3}, set $A=\sum_{j=0}^{k-1}\lambda^{(2)}_{j} \cdot 2^j  + \sum_{j=k}^{k+\alpha-1}2^j$. Reducing modulo $m+1$,
\begin{equation} \nonumber
    A \equiv \modd{2^{k+\alpha}-2^{k+1}+1 \equiv -(2^{k+\alpha}-1)}  {m+1},
\end{equation}
where we have made use of the hypotheses $2^\alpha \equiv \modd{1} {(m+1)/2^\beta}$ and $k \geq \beta$ to substitute $2^{k+\alpha} \equiv \modd{2^k} {m+1}$. Therefore $A$ satisfies the conditions of Lemma \ref{lem2} (with $s=0$, $t=k+\alpha-1$), so we can find $\Lambda^{(2)}_{0}, \ldots, \Lambda^{(2)}_{k+\alpha-1}$ such that $A = \sum_{j=0}^{k+\alpha-1}\Lambda^{(2)}_j$. 

Thus we obtain a presentation
\begin{equation} \nonumber
    2^{k+\alpha}p = 1 + \sum_{j=1}^{k+\alpha}\Lambda^{(1)}_j \cdot 2^j + \sum_{j=0}^{k+\alpha-1}\Lambda^{(2)}_j \cdot 2^j p.
\end{equation}
\end{proof}
It follows that a single $\{-1,m\}$-perfect number of the form $2^kp$ with $k \geq \beta$ and $p$ an odd prime is sufficient to generate infinitely many. The following theorem provides a construction.
\begin{theorem} \label{thrm2}
 Fix $m \geq 1$ and set $\beta = \nu_{2}(m+1)$. Choose $\alpha > \beta$ such that $2^\alpha \equiv \modd{1} {(m+1)/2^\beta}$. If $p \equiv \modd{2(2^{\alpha+1}-1)-1} {2(m+1)}$ is prime, then $2^kp \in \mathcal{P}(-1,m)$ for some $k \geq \alpha$.
\end{theorem}
\begin{proof} Set $N  = 2^{\alpha+1}-1$, and note that $\alpha > \beta$ implies that $N^2 \equiv \modd{1} {2(m+1)}$. If $p \equiv \modd{2(2^{\alpha+1}-1)-1} {2(m+1)}$, then 
\begin{equation} \nonumber
    Np \equiv \modd{2N^2-N=2-N = 3-2^{\alpha+1}} {2(m+1)}.
\end{equation}
That is, $Np-1 \equiv \modd{-2(2^{\alpha}-1)} {2(m+1)}$. Choose $k\geq \alpha$ with $2^k \equiv \modd{2^\alpha} {2(m+1)}$ such that $Np-1 \leq 2m(2^k-1)$. Then by Lemma \ref{lem2}, there are some $\lambda^{(1)}_{1}, \ldots, \lambda^{(1)}_{k} \in \{-1, m\}$ such that 
\begin{equation} \nonumber
    Np = 1 + \sum_{j=1}^{k}\lambda^{(1)}_{j} \cdot 2^j.
\end{equation}
We also have $2^k-N \equiv \modd{-(2^k-1)} {2(m+1)}$, so \begin{equation} \nonumber
    2^k-N = \sum_{j=0}^{k-1}\lambda^{(2)}_{j} \cdot 2^j.
\end{equation}
for some $\lambda^{(2)}_{1}, \ldots, \lambda^{(2)}_{k-1} \in \{-1, m\}$. Therefore
\begin{equation} \nonumber
    1 + \sum_{j=1}^{k}\lambda^{(1)}_{j} \cdot 2^j + \sum_{j=0}^{k-1}\lambda^{(2)}_{j} \cdot 2^jp = Np + (2^k-N)p = 2^kp
\end{equation}
is a presentation.
\end{proof}
\begin{corollary}
There exist infinitely many $\{-1,m\}$-perfect numbers for every $m \geq 1$.
\end{corollary}
\begin{proof}
With $\alpha$, $\beta$ as in Theorem \ref{thrm2}, we have $2(2^{\alpha+1}-1)-1 \equiv \modd{1} {(m+1)/2^\beta}$; then since $2(2^{\alpha+1}-1)-1$ is odd, it follows also that $\text{gcd}(2(2^{\alpha+1}-1)-1,2(m+1))=1$, so there do in fact exist primes $p \equiv \modd{2(2^{\alpha+1}-1)-1} {2(m+1)}$.
\end{proof}

\section{The $\{-1,1\}$-perfect numbers} \label{sec3}
In the previous section, we obtained a characterization of the $\{-1,m\}$-perfect numbers $(m \geq 1)$ of the form $2^kp$ with $p$ an odd prime. When $m=1$, this can be extended to a characterization of all even $\{-1,1\}$-perfect numbers. The $\{-1,1\}$-perfect numbers have a certain aesthetic appeal owing to the formal similarity between the sum involved in a $\{-1,1\}$-presentation and the divisor sum involved in the definition of perfect numbers.

We first refine slightly the relevant special case of Lemma \ref{lem2}.

\begin{lemma} \label{lem3}
If $n \in \mathbb{Z}$, then $n=1+\sum_{j=1}^{k}\lambda_j \cdot 2^j$ for some $k \geq 1$ and $\lambda_1, \ldots, \lambda_k \in \{-1,1\}$ if and only if $n \equiv \modd{3} {4}$.
\end{lemma}
\begin{proof}
Choose $k \geq 1$ such that $-2(2^k-1) \leq n-1 \leq 2(2^k-1)$; then Lemma \ref{lem2} applies ($m=1$, $s=1$, $t=k$).
\end{proof}

\begin{lemma} \label{lem4}
Suppose $n \in \mathbb{Z}$ and let $p$ be prime, with $p$ not dividing $n$. Then $(1)$ if $n \in \mathcal{P}(-1,1)$ , then also  $np^k \in \mathcal{P}(-1,1)$ for all $k \geq 1$, and $(2)$ if $np \in \mathcal{P}(-1,1)$, then also $np^{2k-1} \in \mathcal{P}(-1,1)$ for all $k \geq 1$.
\end{lemma}
\begin{proof}
If $(1)$ $n = \sum{}_1$ and $np^k = \sum{}_2$ $(k \geq 0)$ are presentations of $n$ and $np^k$ respectively, then $np^{k+1}=\sum{}_2 - np^k +p^{k+1}\sum{}_1$ is a presentation of $np^{k+1}$. Similarly, if $(2)$ $np = \sum{}_1$ and $np^k = \sum{}_2$ $(k \geq 1 )$ are presentations of $np$ and $np^k$ respectively, then $np^{k+2}=\sum{}_2 - np^k +p^{k+1}\sum{}_1$ is a presentation of $np^{k+2}$.
\end{proof}
\begin{lemma} \label{lem5}
If $n \in \mathcal{P}(-1,1)$, then also $2n \in \mathcal{P}(-1,1)$.
\end{lemma}
\begin{proof}
If $n$ is odd, this follows from Lemma \ref{lem4}. Suppose $n$ is even and $n=1+\sum_{j=1}^{k}\lambda_{j}d_j$ is a presentation of $n$. Then $2n=1+\sum_{j=1}^{k}\lambda_{j}d_j+n$. The proper divisors of $2n$ missing from this sum have the form $2d_j$ for some divisor $d_j$ of $n$ with $1<d_j<n$ (since $n$ is even). Replace all such $\lambda_{j}d_{j}$ in the sum with $-\lambda_{j}d_j+\lambda_{j}(2d_{j})$ to obtain a presentation of $2n$.
\end{proof}

\begin{theorem} \label{thrm3}
If $d \geq 1$ is odd and not a square, then $2^k d \in \mathcal{P}(-1,1)$ for all but finitely many $k \geq 1$. Conversely if $2^k d \in \mathcal{P}(-1,1)$ for some $k \geq 0, d \geq 1$, then $d$ is not square.
\end{theorem}
\begin{proof}
In light of Lemmas \ref{lem4} and \ref{lem5}, it suffices to show that $2^kp \in \mathcal{P}(-1,1)$ for every odd prime $p$ for some $k \geq 1$. Choose (Lemma \ref{lem3}) $k \geq 1$ and $\lambda_1, \ldots, \lambda_k$ such that
\begin{equation} \nonumber
    1+\sum_{j=1}^{k}\lambda_{j}2^{j} = 
    \begin{cases}
    p, & \text{if } p \equiv \modd{3} {4} \\
    3p, & \text{if } p \equiv \modd{1} {4}
    \end{cases}.
\end{equation}
Then 
\begin{equation} \nonumber
    2^{k}p=1+\sum_{j=1}^{k}\lambda_{j}2^{j}+(-1)^{(p+1)/2}p+\sum_{j=1}^{k-1}2^{j}p
\end{equation}
is a presentation, as required.

Conversely, suppose $n > 1$ is $\{1,-1\}$-perfect with presentation $n=1+\sum_{j=1}^{k}\lambda_{j}d_{j}$. Then
\begin{equation} \nonumber
    \sigma(n) = \sum_{j=1}^{k}(1-\lambda_{j})d_{j}+2n
\end{equation} is even, since every $1-\lambda_{j}=0 \text{ or } 2$. But it is well known that $\sigma(n)$ is even if and only if $n$ is not square or twice a square.
\end{proof}
We conclude with a few further questions and conjectures concerning the $\{-1, 1\}$-perfect numbers. Recall that an abundant number is a positive integer $n$ satisfying $\sigma(n) \geq 2n$. Evidently, every positive $\{-1,1\}$-perfect number is abundant, but not every abundant number is $\{-1,1\}$-perfect; the first few abundant numbers that are not also $\{-1,1\}$-perfect are $18, 20, 36, 72, \ldots$. We conjecture that almost every abundant number is $\{-1,1\}$-perfect.
\begin{conjecture}
The positive $\{-1,1\}$-perfect numbers have a density equal to the density $\mathcal{A}$ of the abundant numbers, for which we have the bounds $0.2474 < \mathcal{A} < 0.2480$, obtained by Del\'eglise \cite{article3}.
\end{conjecture}
The smallest odd abundant number is $945$, which is also $\{-1,1\}$-perfect, as is every other nonsquare odd abundant number that we have been able to check by computer verification. On the other hand, not every odd abundant number is $\{-1,1\}$-perfect: If $n \in \mathbb{Z}$ is odd and abundant, then also $n^2$ is odd and abundant, since the set of abundant numbers is closed under multiplication;  but Theorem \ref{thrm3} shows that $n^2$ cannot be $\{-1,1\}$-perfect. So for example, $945^2=893025$ is odd and abundant, but not $\{-1,1\}$-perfect. We propose the following conjecture.

\begin{conjecture}
Every nonsquare odd abundant number is $\{-1,1\}$-perfect.
\end{conjecture}


\begin{thebibliography}{200}
    \bibitem{article1} Bege, Antal and Fogarasi, Kinga (2009), "Generalized perfect numbers", \textit{Acta Universitas Sapentiae, Mathematica}, \textbf{1}, 1, 73-82

    \bibitem{article2} Cai, Tianxin (2022), "Perfect Numbers and Fibonacci Sequences", \textit{World Scientific Publishing  Co.}, Singapore.
    
    \bibitem{article3} Del\'eglise, Marc (1998), "Bounds for the density of abundant numbers", \textit{Experimental Mathematics} Vol. 7 No. 2, 137-143
    
    \bibitem{article4} Guy, R. K. (1994), "Almost Perfect, Quasi-Perfect, Pseudoperfect, Harmonic, Weird, Multiperfect and Hyperperfect Numbers." \S B2 in \textit{Unsolved Problems in Number Theory}, 2nd ed. New York: Springer-Verlag, 45-53
    
    \bibitem{article5} McCranie, Judson S. (2000), “A Study of Hyperperfect Numbers”, \textit{Journal of Integer Sequences}, Vol. 3, Article 00.1.3
    
    \bibitem{article6} Minoli, Daniel (October 1980), "New Results for hyperperfect numbers", \textit{Abstracts of the American Mathematical Society}, \textbf{1} (6):561

    \bibitem{article7} Minoli, Daniel and Bear, Robert (Fall 1975), "Hyperperfect Numbers", \textit{Pi Mu Epsilon Journal}, 6 (3):153-157
    
    \bibitem{article8} OEIS Foundation Inc.(2021), \textit{The On-Line Encyclopedia of Integer Sequences}, http://oeis.org 
    
    \bibitem{article9} te Riele, Herman J. J. (1981), "Hyperperfect numbers with three different prime factors",  \textit{Mathematics of Computation}, Vol. 36, 297-298
    


\end{thebibliography}
\end{document}